\documentclass[a4paper,12pt]{article}

\usepackage{mathtools}
\usepackage{amsfonts}
\usepackage{amsthm}
\usepackage{graphicx}
\usepackage{framed}
\usepackage{float}
\usepackage{upgreek}
\usepackage{enumerate}
\usepackage{MnSymbol}
\usepackage{amsbsy}
\usepackage{url}
\usepackage{fullpage}

\usepackage[citestyle=alphabetic,bibstyle=alphabetic]{biblatex}
\renewbibmacro{in:}{}
\DefineBibliographyStrings{english}{%
urlseen = {version of},
}
\AtEveryBibitem{
 \clearfield{eprint}
 \clearfield{isbn}
 \clearfield{issn}
 \clearfield{doi}
 \clearfield{language}
 \clearfield{editor}
}
\bibliography{rdsrefs}

\usepackage{bbm}
\usepackage{relsize}
\usepackage{hyperref}
\theoremstyle{plain}
\newtheorem{thm}{Theorem}[section]
\newtheorem{lemma}[thm]{Lemma}
\newtheorem{cor}[thm]{Corollary}

\newtheorem*{theo}{Theorem}

\newtheorem*{sa}{Standing Assumption}
\newtheorem{prop}[thm]{Proposition}

\theoremstyle{definition}
\newtheorem{defi}[thm]{Definition}
\newtheorem{rmk}[thm]{Remark}

\theoremstyle{remark}

\newcommand{\col}{\ensuremath{\hspace{0.3mm}\colon}}
\begin{document}

\title{Synchronisation of almost all trajectories of a random dynamical system}
\author{Julian Newman\footnote{Department of Mathematics, Imperial College London, South Kensington, London, UK, SW7 2AZ. The author wishes gratefully to acknowledge funding from an EPSRC Doctoral Training Account and an EPSRC Doctoral Prize.}}
\maketitle

\begin{abstract}
\noindent It has been shown by Le~Jan that, given a memoryless-noise random dynamical system together with an ergodic distribution for the associated Markov transition probabilities, if the support of the ergodic distribution admits locally asymptotically stable trajectories, then there is a random attracting set consisting of finitely many points, whose basin of forward-time attraction includes a random full measure open set. In this paper, we present necessary and sufficient conditions for this attracting set to be a singleton; our result does not require the state space to be compact, but holds on general Lusin metric spaces.
\end{abstract}

\section{Introduction}

We consider a random dynamical system (RDS) $\varphi$ on a metric space $X$ (which is taken to be a Borel subset of a separable complete metric space), driven by memoryless stationary noise, in either discrete or continuous time. (For a rigorous formulation of this, see Section~\ref{formal}.) Since the noise is stationary and memoryless, the trajectories of $\varphi$ are homogeneous Markov processes. Given an ergodic distribution $\rho$ for the associated Markov transition probabilities, we say that $\varphi$ is \emph{stable with respect to $\rho$} if there is a positive-measure set of noise realisations under which some trajectories in the support of $\rho$ are locally asymptotically stable.\footnote{It follows from this that for almost every realisation of the noise, the trajectory of $\rho$-almost every initial condition in $X$ is asymptotically stable.} When $X$ is a manifold, this property is typically implied by negativity of the Lyapunov spectrum.
\\ \\
It has been shown in \cite{LeJan87} that if $\varphi$ is stable with respect to $\rho$ then there exists $n \in \mathbb{N}$ with the property that for almost every noise realisation, $\rho$-almost all of the state space $X$ can be partitioned into $n$ open (noise-dependent) regions of equal $\rho$-measure, such that trajectories starting in the same region synchronise as time tends to $\infty$ but trajectories starting in different regions do not synchronise as time tends to $\infty$.
\\ \\
In this paper, we will give necessary and sufficient conditions for the number of regions $n$ to be 1; this situation is precisely the situation that for almost every noise realisation, the trajectories of $\rho$-almost all initial conditions in $X$ are locally asymptotically stable and synchronise with each other. We describe this scenario by saying that \emph{$\varphi$ is $\rho$-almost everywhere stably synchronising}.
\\ \\
Let us now describe our result in more detail. It is straightforward to show that if $\varphi$ is $\rho$-almost everywhere stably synchronising then there is a $\rho$-full set $A \subset X$ such that, given any initial conditions $x,y \in A$ and any open $U \subset X$ with $\rho(U)>0$, for almost every noise realisation the trajectories of $x$ and $y$ will (at some point in time) simultaneously be in $U$. Assuming that $\varphi$ is stable with respect to $\rho$, we will show that the converse also holds (that is to say, the existence of such a $\rho$-full set $A \subset X$ implies that $\varphi$ is $\rho$-almost everywhere stably synchronising). Of course, verifying the existence of such a set $A$ is difficult; but we will show that it is sufficient to verify much less than this---namely:

\begin{theo}
Assume $\varphi$ is stable with respect to $\rho$. Suppose there exist a $\rho$-positive measure set $A_1 \subset X$ and a $\rho$-full set $A_2 \subset X$ such that for all $(x,y) \in A_1 \times A_2$ there is a $\rho$-transitive point $p(x,y)$ so that, given any neighbourhood $U \subset X$ of $p(x,y)$, with strictly positive probability the trajectories of $x$ and $y$ will (at some point in time) simultaneously be in $U$. Then $\varphi$ is $\rho$-almost everywhere stably synchronising.
\end{theo}

\noindent Here, a \emph{$\rho$-transitive point} is an initial condition in the support of $\rho$ from which every $\rho$-positive measure open set is accessible. Since $\rho$ is ergodic, $\rho$-almost every initial condition is $\rho$-transitive.
\\ \\
The proof of our result is based on a generalisation of a method in \cite{Hom13}.
\\ \\
To illustrate our result, we will show that the ``double-well potential with additive noise'' as considered in \cite{FGS14} exhibits almost sure forward-time synchronisation of the trajectories of any given pair of initial conditions.\footnote{The results in \cite{FGS14} yield that the noisy double-well potential exhibits global synchronisation in a \emph{pullback} sense. (Nonetheless, combining this with forward-time local asymptotic stability does provide an alternative way to obtain almost sure forward-time synchronisation of the trajectories of any two given initial conditions.)}
\\ \\
Let us now give a brief introduction to synchronisation of trajectories in random dynamical systems, and an overview of existing results on the topic.
\\ \\
Synchronisation of trajectories is manifested physically in the important phenomenon of ``noise-induced synchronisation'', where two or more non-interacting processes starting at different initial states are caused to synchronise with each other due to simultaneous exposure to the same source of external random perturbations. This phenomenon was reported by Pikovskii in the early 1980s (\cite{Pik84}), and since then, there have been numerous case studies of noise-induced synchronisation (analytical, numerical and experimental); see e.g.~\cite{TMHP01} and references therein. The theory of random dynamical systems is central in the mathematical study of noise-induced synchronisation, since the evolutions of the processes affected by the random perturbations can typically be regarded as simultaneous trajectories of one RDS under the same noise realisation.
\\ \\
In analytical studies of synchronisation of trajectories in RDS, a key concept that is often considered is \emph{Lyapunov exponents}. Lyapunov exponents are primarily suited to the context of spatially smooth RDS on Euclidean space or on a manifold, and they measure ``infinitesimal-scale repulsion of trajectories''. When the \emph{maximal Lyapunov exponent} associated to a trajectory exists and is negative, it typically follows that the trajectory is locally asymptotically stable. Given an ergodic distribution $\rho$ for the Markov transition probabilities of the RDS, provided some weak conditions are met, there will exist a value $\lambda_\rho \in \mathbb{R} \cup \{-\infty\}$ such that for almost every noise-realisation, for $\rho$-almost every initial condition the maximal Lyapunov exponent associated to the corresponding trajectory exists and is equal to $\lambda_\rho$. (See e.g.~the start of Section~2 of \cite{LeJan87}.) We refer to $\lambda_\rho$ as the \emph{maximal Lyapunov exponent associated to $\rho$}.
\\ \\
As we have said, negativity of the maximal Lyapunov exponent typically implies \emph{local} asymptotic stability of trajectories; the natural question is then to find conditions under which we can deduce some ``larger-scale'' synchronisation of trajectories. We will now mention some existing results pertaining to this question.
\\ \\
In \cite{newman}, necessary and sufficient conditions are found for a memoryless-noise RDS on a \emph{compact} space to exhibit almost sure synchronisation of the trajectories of any given pair of initial conditions, together with almost sure local asymptotic stability of the trajectory of any given initial condition. One of the key differences between the result of \cite{newman} and the result of our present paper is that, in our present paper, compactness of the state space is not needed. However, it is worth saying that when the state space \emph{is} compact, the necessary and sufficient conditions for ``stable synchronisation'' given in \cite{newman} also serve as \emph{sufficient} conditions for $\rho$-almost-everywhere stable synchronisation; and when these conditions are satisfied, they are likely to be easier to verify than the necessary and sufficient conditions given in this present paper for $\rho$-almost-everywhere stable synchronisation.
\\ \\
In \cite{Hom13}, discrete-time diffeomorphic RDS on a compact manifold are considered. Theorems~1.1\footnote{In the statement of \cite[Theorem~1.1]{Hom13}, it seems that the required additional assumption that $m$ is the only stationary probability measure is missing.} and 1.2 of \cite{Hom13} provide sufficient conditions for almost sure synchronisation of the trajectories of any given pair of initial conditions, in either the whole manifold or a suitable open subset thereof. Theorems~1.1 and 1.2 of \cite{Hom13} can in fact be derived as particular cases of the main result in \cite{newman}. Nonetheless, the basic idea of the proof of \cite[Theorem~1.1]{Hom13} can be generalised well beyond the context of a diffeomorphic RDS on a compact manifold. Specifically, the basic idea of the proof is that, given any set $S$ of initial conditions, if the subsequent trajectories are able to simultaneously reach an arbitrarily small neighbourhood of some point $p$, and if the trajectory starting at $p$ is itself able to reach an open region $U$ within which it is possible for all trajectories to synchronise, then it is possible that the trajectories of all the initial conditions in $S$ will eventually enter $U$ and then synchronise. It is precisely by combining this idea with \cite[Proposition~2]{LeJan87} that the main result of our present paper has been obtained.
\\ \\
In \cite{FGS14}, sufficient conditions are given for a RDS on a separable complete metric space to exhibit mutual ``synchronisation in probability'' of all trajectories.\footnote{More specifically, the phenomena considered in \cite{FGS14} are the existence of a weak global point-attractor and the existence of a weak global attractor.} As an application, large classes of ordinary differential equations in Euclidean space are shown to exhibit such synchronisation when Gaussian white noise is added to the right-hand side.
\\ \\
In \cite{Bax91}, Wiener-driven stochastic differential equations on a compact manifold are considered. Certain non-degeneracy conditions on the vector fields are assumed, implying in particular that there is a unique ergodic distribution $\rho$ for the Markov transition probabilities, and that $\rho$ is equivalent to the Riemannian measure (under any Riemannian metric). One of the results proved (Theorem~4.10) is that if the maximal Lyapunov exponent $\lambda_\rho$ associated to $\rho$ is negative and any two distinct trajectories are always able to come closer together, then the system exhibits almost sure synchronisation of the trajectories of any given pair of initial conditions. This result is, in fact, a special case of the main result of \cite{newman}. However, remarkably, if we replace the condition that $\lambda_\rho$ is negative with the condition that $\lambda_\rho=0$, a further result of \cite{Bax91} (Corollary~5.12) gives that the system will still exhibit a kind of global-scale synchronisation (where the notion of synchronisation involved is based on convergence in probability).
\\ \\
Now there also exist several results to the effect that if a RDS has some order-preserving or orientation-preserving property, then under some weak conditions synchronisation is guaranteed: see e.g.~\cite{CF98} for order-preserving RDS on $\mathbb{R}$, \cite{FGS15} for order-preserving RDS on more general partially ordered spaces, and \cite{newman2} and \cite{Kai93} for orientation-preserving RDS on a circle.
\\ \\
The structure of the paper will be as follows: In Section~2, we will present the formal setup, introduce some key definitions and results, and then state our main result (Theorem~\ref{main}). We will also present the double-well potential example. In Section~3, we give the proof of our main result, first introducing some preliminary theory of RDS as necessary.

\section{The basic setup and our result} \label{formalism}

\subsection{The setup: RDS with memoryless noise} \label{formal}

A ``random dynamical system with memoryless noise'' consists of two components: a ``memoryless'' filtered measure-preserving flow, representing the ``noise''; and an adapted cocycle over this flow acting on the state space.
\\ \\
Let $\mathbb{T}$ be either $\mathbb{Z}$ or $\mathbb{R}$, and let $\mathbb{T}^+:=\mathbb{T} \cap [0,\infty)$. Let $\bar{\mathbb{T}}:=\mathbb{T} \cup \{-\infty,\infty\}$, and let $\bar{\mathbb{T}}^+:=\mathbb{T}^+\cup\{\infty\}$. Let $(\Omega,\mathcal{F})$ be a measurable space, and let $(\mathcal{F}_s^{s+t})_{s \in \mathbb{T} \! , \, t \in \mathbb{T}^+}$ be a family of sub-$\sigma$-algebras of $\mathcal{F}$ such that
\begin{enumerate}[\indent (i)]
\item $\mathcal{F}_{t_1}^{t_2} \,\subset\, \mathcal{F}_{t_0}^{t_3}\,$ for all $\,t_0 \leq t_1 \leq t_2 \leq t_3\,$ in $\mathbb{T}$;
\item $\sigma(\mathcal{F}_s^{s+t} : s \in \mathbb{T},t \in \mathbb{T}^+)=\mathcal{F}$.
\end{enumerate}
\noindent We will use the following notations:
\begin{align*}
\mathcal{F}_s^\infty \ &:= \ \sigma(\mathcal{F}_s^{s+t} : t \in \mathbb{T}^+) \hspace{3mm} \textrm{for any $s \in \mathbb{T}$} \\
\mathcal{F}_\infty^\infty \ &:= \ \bigcap_{s \in \mathbb{T}} \mathcal{F}_s^\infty \\
\mathcal{F}_{-\infty}^t \ &:= \ \sigma(\mathcal{F}_{t-s}^t : s \in \mathbb{T}^+) \hspace{3mm} \textrm{for any $t \in \mathbb{T}$} \\
\mathcal{F}_{-\infty}^{-\infty} \ &:= \ \bigcap_{t \in \mathbb{T}} \mathcal{F}_{-\infty}^t
\end{align*}
\noindent It will also be useful to have the convention that $\mathcal{F}_{-\infty}^\infty:=\mathcal{F}$. Let $(\theta^t)_{t \in \mathbb{T}}$ be a group of $(\mathcal{F},\mathcal{F})$-measurable functions $\theta^t\col\Omega \to \Omega$ such that $\theta^\tau\mathcal{F}_s^t=\mathcal{F}_{s-\tau}^{t-\tau}$ for all $s,t,\tau \in \mathbb{T}$ with $s \leq t$. Let $\mathbb{P}$ be a probability measure on $(\Omega,\mathcal{F})$ with the following properties:
\begin{enumerate}[\indent (i)]
\item $\theta^t_\ast\mathbb{P}=\mathbb{P}$ for all $t \in \mathbb{T}$;
\item for each $t \in \mathbb{T}$, $\mathcal{F}_{-\infty}^t$ and $\mathcal{F}_t^\infty$ are independent $\sigma$-algebras under $\mathbb{P}$.
\end{enumerate}

\noindent Property~(i) represents stationarity of the noise, and property~(ii) represents memorylessness of the noise. As in \cite[Lemma~5.1]{newman2}, for every $t \in \mathbb{T} \setminus \{0\}$, $\mathbb{P}$ is ergodic with respect to $\theta^t$.
\\ \\
Let $(X,d)$ be a separable metric space such that $X$ is a Borel subset of the $d$-completion of $X$.\footnote{This guarantees that $X$ is measurably isomorphic to either an at-most-countable discrete space, or an interval with its Borel $\sigma$-algebra (\cite[Theorem~3.3.13]{MR1619545}).} For any $x \in X$ and $\delta>0$, we write $B_\delta(x):=\{y \in X : d(x,y)<\delta\}$. For any $A \subset X$, we write $\Delta_A:=\{(x,x) : x \in A\} \subset X \times X$.
\\ \\
Let $\,\varphi = \left(\varphi(t,\omega)\right)_{t \in \mathbb{T}^+ \! , \, \omega \in \Omega}\,$ be a $(\mathbb{T}^+ \! \times \Omega)$-indexed family of continuous functions $\varphi(t,\omega):X \to X$ such that
\begin{enumerate}[\indent (a)]
\item the map $\omega \mapsto \varphi(t,\omega)x$ is $(\mathcal{F}_t,\mathcal{B}(X))$-measurable for each $t \in \mathbb{T}^+$ and $x \in X$;
\item for every $\omega \in \Omega$, $\varphi(0,\omega)$ is the identity function on $X$;
\item $\varphi(s+t,\omega) \, = \, \varphi(t,\theta^s\omega) \circ \varphi(s,\omega)\,$ for all $s,t \in \mathbb{T}^+$ and $\omega \in \Omega$;
\item for any decreasing sequence $(t_n)$ in $\mathbb{T}^+$ converging to a value $t$, and any sequence $(x_n)$ in $X$ converging to a point $x$, $\,\varphi(t_n,\omega)x_n \to \varphi(t,\omega)x\,$ as $n \to \infty$ for all $\omega \in \Omega$.
\end{enumerate}

\noindent (Property~(d) constitutes ``right-continuity'' of $\varphi$.\footnote{As shown in the Appendix of \cite{newman}, assuming an additional ``left limits'' property guarantees that local asymptotic stability (in the sense that we shall use the term) implies stability in the sense of Lyapunov.})
\\ \\
We refer to $\varphi$ as a \emph{random dynamical system on the state space $X$ over the noise space $(\Omega,\mathcal{F},(\mathcal{F}_s^{s+t})_{s \in \mathbb{T} \! , \, t \in \mathbb{T}^+},\mathbb{P},(\theta^t)_{t \in \mathbb{T}})$.}
\\ \\
Now it is easy to show that for any $x \in X$, the stochastic process $\left( \varphi(t,\cdot)x \right)_{t \in \mathbb{T}^+}$ is a homogeneous Markov process (with respect to the filtration $(\mathcal{F}_0^t)_{t \in \mathbb{T}^+}$), with the associated family of transition probabilities $(\varphi_x^t)_{x \in X \! , \, t \in \mathbb{T}^+}$ being given by
\begin{align*}
\varphi_x^t(A) \ :=& \ \mathbb{P}(\omega : \varphi(t,\omega)x \in A) \\
=& \ \mathbb{P}(\omega : \varphi(t,\theta^s\omega)x \in A) \hspace{4mm} \textrm{(for any $s \in \mathbb{T}$)}
\end{align*}
\noindent for all $A \in \mathcal{B}(X)$. Note that a probability measure $\rho$ on $X$ is a stationary probability measure of the Markov transition probabilities $(\varphi_x^t)_{x \in X \! , \, t \in \mathbb{T}^+}$ if and only if for all $A \in \mathcal{B}(X)$ and $t \in \mathbb{T}^+$,
\[ \rho(A) \ = \ \int_\Omega \rho(\varphi(t,\omega)^{-1}(A)) \, \mathbb{P}(d\omega). \]
\noindent For any $t \in \mathbb{T}^+$ we define the map $\Theta^t:\Omega \times X \to \Omega \times X$ by 
\[ \Theta^t(\omega,x) \ = \ (\theta^t\omega,\varphi(t,\omega)x). \]
\noindent Note that $(\Theta^t)_{t \in \mathbb{T}^+}$ forms a semigroup of measurable transformations of the measurable space $(\Omega \times X, \mathcal{F} \otimes \mathcal{B}(X))$, and also of the ``restricted'' measurable space $(\Omega \times X, \mathcal{F}_{-r}^\infty \otimes \mathcal{B}(X))$ for any $r \in \mathbb{T}^+$. For any Borel probability measure $\rho$ on $X$, the following hold:
\begin{itemize}
\item $\rho$ is a stationary measure of the Markov transition probabilities $(\varphi_x^t)_{x \in X \! , \, t \in \mathbb{T}^+}$ if and only if $(\Theta^t)_{t \in \mathbb{T}^+}$ is a measure-preserving semigroup of the probability space $(\Omega \times X, \mathcal{F}_0^\infty \otimes \mathcal{B}(X),\mathbb{P}|_{\mathcal{F}_0^\infty} \otimes \rho)$;
\item $\rho$ is an ergodic measure of the Markov transition probabilities $(\varphi_x^t)_{x \in X \! , \, t \in \mathbb{T}^+}$ if and only if $(\Theta^t)_{t \in \mathbb{T}^+}$ is an ergodic measure-preserving semigroup of the probability space $(\Omega \times X, \mathcal{F}_0^\infty \otimes \mathcal{B}(X),\mathbb{P}|_{\mathcal{F}_0^\infty} \otimes \rho)$.
\end{itemize}

\noindent (For a proof, see e.g.~\cite[Theorem~143]{New15} or \cite[Lemma~I.2.3 and Theorem~I.2.1]{Kif86}.)

\subsection{Stability of trajectories and our main result}

We now introduce the notion of asymptotic stability; we then give (a generalised version of) an important result in \cite{LeJan87}, and from there, state our main result.
\\ \\
Given a sample point $\omega \in \Omega$ and a set $A \subset X$, we say that $A$ \emph{contracts under $\omega$} if $\mathrm{diam}(\varphi(t,\omega)A) \to 0$ as $t \to \infty$. Given a sample point $\omega \in \Omega$ and a point $x \in X$, we say that \emph{$x$ is asymptotically stable under $\omega$} if there exists a neighbourhood $U$ of $x$ such that $U$ contracts under $\omega$. We say that a set $A \subset X$ \emph{admits stable trajectories} if
\[ \mathbb{P}( \omega \, : \, \exists \,\textrm{open $U$ with $U \cap A \neq \emptyset$ s.t.~$U$ contracts under $\omega$} ) \ > \ 0, \]
\noindent which is the same as saying that
\[ \mathbb{P}( \omega \, : \, \exists \,\textrm{$x \in A$ s.t.~$x$ is asymptotically stable under $\omega$} ) \ > \ 0. \]
\noindent Now let
\[ O \ := \ \{ (\omega,x) \in \Omega \times X : \textrm{$x$ is asymptotically stable under $\omega$} \}. \]
\noindent As in \cite[Lemma~2.2.3]{newman}, $O$ is an $(\mathcal{F}_0^\infty \otimes \mathcal{B}(X))$-measurable set, and is backward-invariant under the semigroup $(\Theta^t)_{t \in \mathbb{T}^+}$.

\begin{lemma} \label{stable traj}
Let $\rho$ be an ergodic probability measure of the Markov transition probabilities $(\varphi_x^t)$. The following statements are equivalent:
\begin{enumerate}[\indent (i)]
\item $O$ is a $(\mathbb{P} \otimes \rho)$-full measure set;
\item $O$ is a $(\mathbb{P} \otimes \rho)$-positive measure set;
\item $\mathrm{supp}\,\rho$ admits stable trajectories.
\end{enumerate}
\end{lemma}

\begin{proof}
The equivalence of (i) and (ii) follows from the backward-invariance of $O$ and the fact that $\mathbb{P}|_{\mathcal{F}_0^\infty} \otimes \rho$ is $(\Theta^t)$-ergodic. It is clear that (ii)$\Rightarrow$(iii). Now suppose that (iii) holds; so we have a $\mathbb{P}$-positive measure set of sample points $\omega$ with the property that there exists a $\rho$-positive measure open set $U$ such that $U$ contracts under $\omega$. Fubini's theorem then yields that $\mathbb{P} \otimes \rho(O)>0$, i.e.~(ii) holds.
\end{proof}

\begin{defi}
Let $\rho$ be an ergodic probability measure of $(\varphi_x^t)$. We say that \emph{$\varphi$ is stable with respect to $\rho$} if the equivalent statements in Lemma~\ref{stable traj} hold.
\end{defi}

\noindent Now given a sample point $\omega \in \Omega$ and an open set $U \subset X$, we will say that \emph{$U$ is $\sigma$-contracting under $\omega$} if there exists an increasing sequence $U_1 \subset U_2 \subset U_3 \subset \ldots$ of open subsets of $X$ such that $U=\bigcup_{i=1}^\infty U_i$ and $U_i$ contracts under $\omega$ for each $i \in \mathbb{N}$.

\begin{defi}
Let $\rho$ be an ergodic probability measure of $(\varphi_x^t)$. We say that \emph{$\varphi$ is $\rho$-almost everywhere stably synchronising} if for $\mathbb{P}$-almost every $\omega \in \Omega$ there exists a $\rho$-full measure open set that is $\sigma$-contracting under $\omega$.
\end{defi}

\begin{rmk} \label{pairwise sync}
It is not hard to show that $\varphi$ is $\rho$-almost everywhere stably synchronising if and only if the following statements both hold:
\begin{enumerate}[\indent (i)]
\item $\varphi$ is stable with respect to $\rho$;
\item there is a $\rho$-full set $A \subset X$ such that for all $x,y \in A$,
\[ \mathbb{P}( \omega \, : \, d(\varphi(t,\omega)x,\varphi(t,\omega)y) \to 0 \textrm{ as } t \to \infty ) \ = \ 1. \]
\end{enumerate}
\end{rmk}

\noindent The following is a generalised statement of \cite[Proposition~3]{LeJan87}:

\begin{prop} \label{le jan}
Let $\rho$ be an ergodic probability measure of $(\varphi_x^t)$, and suppose that $\varphi$ is stable with respect to $\rho$. Then there exists $n_\rho \in \mathbb{N}$ such that the following holds: for $\mathbb{P}$-almost every $\omega \in \Omega$, there exist open sets $U_1(\omega),\ldots,U_{n_\rho}(\omega)$ with $\rho(U_i(\omega))=\frac{1}{n_\rho}$ for each $1 \leq i \leq n_\rho$ such that
\begin{itemize}
\item $U_i(\omega)$ is $\sigma$-contracting under $\omega$ for each $1 \leq i \leq n_\rho$, and
\item for every $1 \leq i < j \leq n_\rho$, for every $x \in U_i(\omega)$ and $y \in U_j(\omega)$, $d(\varphi(t,\omega)x,\varphi(t,\omega)y)$ does not tend to $0$ as $t \to \infty$.
\end{itemize}
\end{prop}

\noindent (The proof will be outlined in Section~3.)
\\ \\
So the situation that $\varphi$ is $\rho$-almost everywhere stably synchronising is precisely the situation that $n_\rho=1$. We go on to present our new sharp criteria for this situation.

\begin{defi}
Given points $x,y,p \in X$, we will say that \emph{$(x,y)$ is contractible towards $p$} if for every $\varepsilon>0$,
\[ \mathbb{P}( \, \omega \, : \, \exists \, t \in \mathbb{T}^+ \textrm{ s.t.~} (\varphi(t,\omega)x,\varphi(t,\omega)y) \in B_\varepsilon(p) \times B_\varepsilon(p) \, ) \ > \ 0. \]
\end{defi}

\noindent Since the map $t \mapsto \varphi(t,\omega)u$ is right-continuous for all $u \in X$, this is equivalent to saying that there exists $t \in \mathbb{T}^+ \cap \mathbb{Q}$ such that
\[ \mathbb{P}( \, \omega \, : \, (\varphi(t,\omega)x,\varphi(t,\omega)y) \in B_\varepsilon(p) \times B_\varepsilon(p) \, ) \ > \ 0. \]

\begin{defi} \label{contractible to set}
Given points $x,y \in X$ and a set $A \subset X$, we will say that \emph{$(x,y)$ is contractible towards $A$} if for every neighbourhood $V$ of $\Delta_A$ in $X \times X$,
\[ \mathbb{P}( \, \omega \, : \, \exists \, t \in \mathbb{T}^+ \textrm{ s.t.~} (\varphi(t,\omega)x,\varphi(t,\omega)y) \in V \, ) \ > \ 0. \]
\end{defi}

\begin{lemma} \label{exists point}
For any $x,y \in X$ and $A \subset X$, $(x,y)$ is contractible towards $A$ if and only if there exists $p \in A$ such that $(x,y)$ is contractible towards $p$.
\end{lemma}

\begin{proof}
It is clear that if there exists $p \in A$ such that $(x,y)$ is contractible towards $p$, then $(x,y)$ is contractible towards $A$. Now suppose there does not exist $p \in A$ such that $(x,y)$ is contractible towards $p$. Let
\begin{align*}
\mathcal{U} \ :=& \ \{ \, \textrm{open } V \subset X \times X \, : \, \mathbb{P}( \, \omega \, : \, \exists \, t \in \mathbb{T}^+ \textrm{ s.t.~} (\varphi(t,\omega)x,\varphi(t,\omega)y) \in V \, ) \, = \, 0 \, \} \\
=& \ \{ \, \textrm{open } V \subset X \times X \, : \, \textrm{for all } t \in \mathbb{T}^+ \cap \mathbb{Q}, \ \mathbb{P}( \, \omega \, : \, (\varphi(t,\omega)x,\varphi(t,\omega)y) \in V \, ) \, = \, 0 \, \}
\end{align*}
\noindent and let $W:=\bigcup_{V \in \mathcal{U}} V$. For every $p \in A$, since $(x,y)$ is not contractible towards $p$, there exists $\varepsilon>0$ such that $B_\varepsilon(p) \times B_\varepsilon(p) \subset W$. Hence $\Delta_A \subset W$. Now since $X \times X$ is second-countable, there exists a countable subcollection $\mathcal{V}$ of $\mathcal{U}$ such that $W=\bigcup_{V \in \mathcal{V}} V$. It therefore follows in particular that for every $t \in \mathbb{T}^+ \cap \mathbb{Q}$,
\[ \mathbb{P}( \, \omega \, : \, (\varphi(t,\omega)x,\varphi(t,\omega)y) \in W \, ) \ = \ 0. \]
\noindent Hence
\[ \mathbb{P}( \, \omega \, : \, \exists \, t \in \mathbb{T}^+ \textrm{ s.t.~} (\varphi(t,\omega)x,\varphi(t,\omega)y) \in W \, ) \ = \ 0. \]
\noindent So $(x,y)$ is not contractible towards $A$.
\end{proof}

\noindent For any $p \in X$, we will write $\mathfrak{C}_p \subset X \times X$ for the set of pairs that are contractible towards $p$. For any $A \subset X$, we will write $\mathfrak{C}_A \subset X \times X$ for the set of pairs that are contractible towards $A$.

\begin{defi}
Let $\rho$ be an ergodic probability measure of $(\varphi_x^t)$. We will say that a point $x \in \mathrm{supp}\,\rho$ is \emph{$\rho$-transitive} if for every open $U \subset X$ with $\rho(U)>0$,
\[ \mathbb{P}( \, \omega \, : \, \exists \, t \in \mathbb{T}^+ \textrm{ s.t.~} \varphi(t,\omega)x \in U \, ) \ > \ 0. \]
\end{defi}
\noindent This is equivalent to saying that for some $t \in \mathbb{T}^+ \cap \mathbb{Q}$, $\varphi_x^t(U)>0$. We will write $A_\rho$ for the set of points in $\mathrm{supp}\,\rho$ that are $\rho$-transitive.
\\ \\
By the ergodic theorem for Markov processes,\footnote{See e.g.~\cite[Corollary~57]{New15}, with $Y$ being the set of right-continuous paths in $X$.} $\rho$-almost every $x \in \mathrm{supp}\,\rho$ has the property that for $\mathbb{P}$-almost all $\omega \in \Omega$, for every $T \in \mathbb{T}^+$, $\{\varphi(t,\omega)x:t \geq T\}$ is dense in $\mathrm{supp}\,\rho$. Hence in particular, $\rho(A_\rho)=1$.

\begin{defi} \label{rho rectangle}
Let $\rho$ be a probability measure on $X$. A \emph{$\rho$-full-length rectangle} is a set $A \subset X \times X$ taking the form $A=A_1 \times A_2$ where $A_1,A_2 \in \mathcal{B}(X)$ with $\rho(A_1)>0$ and $\rho(A_2)=1$.
\end{defi}

\noindent Our main result is the following:

\begin{thm} \label{main}
Let $\rho$ be an ergodic probability measure of $(\varphi_x^t)$, and suppose that $\varphi$ is stable with respect to $\rho$. The following statements are equivalent:
\begin{enumerate}[\indent (i)]
\item there is a non-$\rho$-null set $R \subset X$ such that for each $p \in R$, the set $\mathfrak{C}_p$ contains a $\rho$-full-length rectangle;
\item the set $\mathfrak{C}_{A_\rho}$ contains a $\rho$-full-length rectangle;
\item $\varphi$ is $\rho$-almost everywhere stably synchronising;
\item there is a $\rho$-full set $A \subset \mathrm{supp}\,\rho$ such that given any $x,y \in A$, $\mathbb{P}$-almost every $\omega \in \Omega$ has the property that for any open $U \subset X$ with $\rho(U)>0$ there exists $t \in \mathbb{T}^+$ such that $\varphi(t,\omega)x,\varphi(t,\omega)y \in U$.
\end{enumerate}
\end{thm}

\noindent Let us now consider the example of a ``double-well potential perturbed by Gaussian white noise''. Fix an integer $d \geq 2$. Let $\mathbb{T}=\mathbb{R}$. Let $\Omega:=\{\omega \in C(\mathbb{R},\mathbb{R}^d) : \omega(0)=\mathbf{0}\}$, let $\mathcal{F}$ be the smallest $\sigma$-algebra on $\Omega$ with respect to which the projections $\omega \mapsto \omega(t)$ are measurable for all $t \in \mathbb{R}$, let $\mathbb{P}$ be the Wiener measure on $(\Omega,\mathcal{F})$, and for each $\tau \in \mathbb{R}$ let $\theta^\tau\col \Omega \to \Omega$ be given by $(\theta^\tau\omega)(t)=\omega(t+\tau)-\omega(\tau)$. Let $X=\mathbb{R}^d$ (equipped with the Euclidean metric). As in \cite{FGS14}, let $\varphi$ be such that for all $\omega \in \Omega$ and $x \in \mathbb{R}^d$, the function $u(t)=\varphi(t,\omega)x$ is the solution of the integral equation
\begin{align*}
u(t) \ &= \ x + \int_0^t (1-|u(s)|^2)u(s)\,ds + \omega(t) \\
&= \ x + \int_0^t b(u(s))\,ds + \omega(t)
\end{align*}
\noindent where $b(y):=(1-|y|^2)y$ for all $y \in \mathbb{R}^d$. In other words, $\varphi$ is the ``RDS generated by the stochastic differential equation''
\[ du_t \ = \ (1-|u_t|^2)u_t\,dt \; + \, dW_t. \]
\noindent It is not hard to show (e.g.~by computing explicitly the Jacobian of $b$) that, as in \cite{FGS14}, $b$ satisfies the \emph{one-sided Lipschitz condition}---that is to say, there exists $L \in \mathbb{R}$ such that for all $y_1,y_2 \in \mathbb{R}^d$,
\[ ( b(y_2) - b(y_1) ) \boldsymbol{\cdot} (y_2-y_1) \ \leq \ L|y_2 - y_1|^2. \]
\noindent Now for any $\omega \in \Omega$, $x_1,x_2 \in \mathbb{R}^d$ and $t_0 \in \mathbb{R}$, if we let $u_1(t):=\varphi(t,\theta^{t_0}\omega)x_1$ and $u_2(t):=\varphi(t,\theta^{t_0}\omega)x_2$ for all $t \geq 0$, we find that
\[ u_2(t) - u_1(t) \ = \ (u_2(0) - u_1(0)) \, + \int_0^t b(u_2(s)) - b(u_1(s)) \, ds \]
\noindent and therefore
\[ \left. \frac{d}{d\tau} |u_2(\tau) - u_1(\tau)|^2 \right|_{\tau=t} \ = \ 2 ( b(u_2(t)) - b(u_1(t)) ) \boldsymbol{\cdot} (u_2(t) - u_1(t)) \ \leq \ 2L|u_2(t) - u_1(t)|^2. \]
\noindent Gr\"{o}nwall's inequality then gives that
\[ |u_2(t) - u_1(t)| \ \leq \ |u_2(0)-u_1(0)|e^{Lt}. \]
\noindent In other words, for all $\omega \in \Omega$, $x_1,x_2 \in \mathbb{R}^d$, and $t_0,t_1 \in \mathbb{R}$ with $t_1 \geq t_0$, we have
\[ |\varphi(t_1,\omega)x_2 \, - \, \varphi(t_1,\omega)x_1| \ \leq \ e^{L(t_1-t_0)}|\varphi(t_0,\omega)x_2 \, - \, \varphi(t_0,\omega)x_1|. \]
\noindent As a consequence, we have that for any $A \subset \mathbb{R}^d$ and any $\omega \in \Omega$,
\begin{equation} \label{d-c}
\mathrm{diam}(\varphi(n,\omega)A) \to 0 \textrm{ as $n \to \infty$ in $\mathbb{N}$} \hspace{3mm} \Longrightarrow \hspace{3mm} \textrm{$A$ contracts under $\omega$.}
\end{equation}
\noindent Now as in \cite{FGS14}, there exists a unique $(\varphi_x^t)$-ergodic probability measure $\rho$ on $\mathbb{R}^d$, and $\rho$ has full support. By \cite[Example~4.8]{FGS14}, the maximal Lyapunov exponent associated to $\rho$ is negative. As is described in Section~4 of \cite{FGS14}, it follows that for $(\mathbb{P} \otimes \rho)$-almost all $(\omega,x) \in \Omega \times \mathbb{R}^d$, there is a neighbourhood $U$ of $x$ such that $\mathrm{diam}(\varphi(n,\omega)U) \to 0$ as $n \to \infty$ in $\mathbb{N}$; so (\ref{d-c}) then gives that $\varphi$ is stable with respect to $\rho$. Now (as with any additive-noise SDE) one can show that every point in $\mathbb{R}^d$ is $\rho$-transitive: Fix any $x \in \mathbb{R}^d$ and any non-empty open $U \subset \mathbb{R}^d$; take any $y \in U$ and, selecting a sufficiently large value $\eta_0>0$, take a sample point $\omega_0 \in \Omega$ with
\[ \omega_0(t) \ = \ \eta_0t(y-x) \hspace{3mm} \forall \, t \in [0,\tfrac{1}{\eta_0}]. \]
\noindent Then we will have that $\varphi(\frac{1}{\eta_0},\omega_0) \in U$. Since the Wiener measure $\mathbb{P}$ has full support in the topology of uniform convergence on compact sets (\cite[Proposition~477F]{Fre13}), it follows that $\varphi_x^{1\!/\!\eta_0}(U)>0$. Since $U$ was arbitrary, $x$ is $\rho$-transitive. Now it is not hard to see that every $(x,y) \in \mathbb{R}^d \times \mathbb{R}^d$ is contractible towards the point $(1,\mathbf{0}) \in \mathbb{R}^d$: Fixing any $\varepsilon>0$, we can select sufficiently large values $\eta_1,\eta_2>0$ that if we take a sample point $\omega_1$ with
\[ \omega_1(t) \ = \ \left\{ \!\! \begin{array}{c l} (\eta_1\eta_2t,\mathbf{0}) & t \in [0,\frac{1}{\eta_1}] \\ (\eta_2,\mathbf{0}) & t \in [\frac{1}{\eta_1},\infty), \end{array} \right. \]
\noindent we will have that $\,\varphi(t,\omega_1)x, \, \varphi(t,\omega_1)y \in B_\varepsilon((1,\mathbf{0}))\,$ for all sufficiently large $t$; so once again, since $\mathbb{P}$ has full support, it follows that $(x,y)$ is contractible towards $(1,\mathbf{0})$. So then, $\varphi$ satisfies hypothesis~(ii) of Theorem~\ref{main} (since $A_\rho=\mathbb{R}^d$ and $\mathfrak{C}_{\mathbb{R}^d} \supset \mathfrak{C}_{(1,\mathbf{0})} = \mathbb{R}^d \times \mathbb{R}^d$), and therefore $\varphi$ is $\rho$-almost stably synchronising. Now by Remark~\ref{pairwise sync}, there exists a $\rho$-full set $A \subset X$ such that for all $x \in A$,
\[ \mathbb{P}( \omega \, : \,  \textrm{$x$ is asymptotically stable under $\omega$} ) \ = \ 1 \]
\noindent and for all $x,y \in A$,
\[ \mathbb{P}( \omega \, : \, d(\varphi(t,\omega)x,\varphi(t,\omega)y) \to 0 \textrm{ as } t \to \infty ) \ = \ 1. \]
\noindent Now for every $x \in X$ and $t>0$, $\varphi_x^t$ is equivalent to the Lebesgue measure, and also the stationary measure $\rho$ is equivalent to the Lebesgue measure. Hence, for every $x \in X$ we have that $\varphi_x^1(A)=1$, and for all $x,y \in X$ we have that
\[ \mathbb{P}( \omega \, : \, \varphi(1,\omega)x,\varphi(1,\omega)y \in A ) \ = \ 1. \]
\noindent Consequently, due to the memorylessness of the noise, we can conclude that for \emph{all} $x \in X$,
\[ \mathbb{P}( \omega \, : \,  \textrm{$x$ is asymptotically stable under $\omega$} ) \ = \ 1, \]
\noindent and for \emph{all} $x,y \in X$,
\[ \mathbb{P}( \omega \, : \, d(\varphi(t,\omega)x,\varphi(t,\omega)y) \to 0 \textrm{ as } t \to \infty ) \ = \ 1. \]

\section{Invariant measures and the proof of Theorem~\ref{main}} \label{proof section}

We start by introducing some basic theory of invariant measures of random dynamical systems. We define the projections $\pi_\Omega\col \Omega \times X \to \Omega$ and $\pi_X\col \Omega \times X \to X$ by $\pi_\Omega(\omega,x)=\omega$ and $\pi_X(\omega,x)=x$.
\\ \\
A \emph{random probability measure on $X$} is an $\Omega$-indexed family $(\mu_\omega)_{\omega \in \Omega}$ of probability measures on $X$ such that the map $\omega \mapsto \mu_\omega(A)$ is measurable for all $A \in \mathcal{B}(X)$. We will say that two random probability measures $(\mu_\omega^1)_{\omega \in \Omega}$ and $(\mu_\omega^2)_{\omega \in \Omega}$ are \emph{equivalent} if for $\mathbb{P}$-almost all $\omega \in \Omega$, $\mu_\omega^1=\mu_\omega^2$. For any random probability measure $(\mu_\omega)$, we may define a probability measure $\mu$ on the product space $(\Omega \times X, \mathcal{F} \otimes \mathcal{B}(X))$ by
\[ \mu(A) \ = \ \int_\Omega \mu_\omega(A_\omega) \, \mathbb{P}(d\omega) \hspace{3mm} \forall \, A \in \mathcal{B}(X) \]
\noindent where $A_\omega:=\{x \in X : (\omega,x) \in A\}$. We refer to $\mu$ as the \emph{integrated form of $(\mu_\omega)$}. Note that two equivalent random probability measures share the same integrated form.
\\ \\
We will say that a random probability measure $(\mu_\omega)$ is \emph{atomless} if for $\mathbb{P}$-almost every $\omega \in \Omega$, $\mu_\omega$ is an atomless probability measure (i.e.~$\mu_\omega(\{x\})=0$ for all $x \in X$). Given an integer $n \in \mathbb{N}$, we will say that a random probability measure $(\mu_\omega)$ is \emph{$n$-uniform} if for $\mathbb{P}$-almost every $\omega \in \Omega$ there exist distinct points $x_1,\ldots,x_n \in X$ such that $\mu_\omega = \frac{1}{n} \sum_{i=1}^n \delta_{x_i}$.

\begin{rmk} \label{diagonal measure}
Given a random probability measure $(\mu_\omega)$, let $\bar{\mu}$ be the probability measure on $X$ given by
\[ \bar{\mu}(A) \ = \ \int_\Omega \mu_\omega(A) \, \mathbb{P}(d\omega) \]
\noindent for all $A \in \mathcal{B}(X)$, and let $\bar{\mu}^{(2)}$ be the probability measure on $X \times X$ given by
\[ \bar{\mu}^{(2)}(A) \ = \ \int_\Omega \mu_\omega \otimes \mu_\omega(A) \, \mathbb{P}(d\omega) \]
\noindent for all $A \in \mathcal{B}(X \times X)$. (Note that $\bar{\mu}$ is precisely $\pi_{X\ast}\mu$, where $\mu$ is the integrated form of $(\mu_\omega)$.) By Fubini's theorem, if $(\mu_\omega)$ is atomless then $\bar{\mu}^{(2)}(\Delta_X)=0$, and if $(\mu_\omega)$ is $n$-uniform then for all $A \in \mathcal{B}(X)$, $\bar{\mu}^{(2)}(\Delta_A)=\frac{1}{n}\bar{\mu}(A)$.
\end{rmk}

\noindent Now we will say that a probability measure $\mu$ on the product space $(\Omega \times X, \mathcal{F} \otimes \mathcal{B}(X))$ is \emph{$\mathbb{P}$-compatible} if $\pi_{\Omega\ast}\mu=\mathbb{P}$. It is clear that the integrated form of a random probability measure is itself a $\mathbb{P}$-compatible probability measure. The \emph{disintegration theorem} (\cite[Proposition~3.6]{MR1993844}) states that for any $\mathbb{P}$-compatible probability measure $\mu$ there exists a random probability measure $(\mu_\omega)$ whose integrated form coincides with $\mu$, and this random probability measure is unique up to equivalence; we refer to $(\mu_\omega)$ as \emph{a (version of the) disintegration} of $\mu$. We say that a $\mathbb{P}$-compatible probability measure $\mu$ is \emph{past-measurable} if $\mu$ admits a disintegration $(\mu_\omega)$ such that the map $\omega \mapsto \mu_\omega(A)$ is $\mathcal{F}_{-\infty}^0$-measurable for all $A \in \mathcal{B}(X)$. It is easy to show (using the fact that $\mathcal{F}_{-\infty}^0$ and $\mathcal{F}_0^\infty$ are independent $\sigma$-algebras) that for any past-measurable $\mathbb{P}$-compatible probability measure $\mu$, the restriction of $\mu$ to $\mathcal{F}_0^\infty \otimes \mathcal{B}(X)$ coincides with $\mathbb{P}|_{\mathcal{F}_0^\infty} \otimes \pi_{X\ast}\mu$.
\\ \\
We will say that a probability measure $\mu$ on $(\Omega \times X, \mathcal{F} \otimes \mathcal{B}(X))$ is \emph{an invariant measure of $\varphi$} if $\mu$ is both $\mathbb{P}$-compatible and invariant under the semigroup $(\Theta^t)_{t \in \mathbb{T}^+}$. It is not hard to show that a $\mathbb{P}$-compatible probability measure $\mu$ with disintegration $(\mu_\omega)$ is invariant under $\varphi$ if and only if
\begin{equation} \label{inv}
\mathbb{P}( \omega \in \Omega : \mu_{\theta^t\omega} = \varphi(t,\omega)_\ast\mu_\omega) = 1 \hspace{3mm} \forall \, t \in \mathbb{T}^+.
\end{equation}
\noindent We will say that a probability measure $\mu$ on $(\Omega \times X, \mathcal{F} \otimes \mathcal{B}(X))$ is \emph{an ergodic measure of $\varphi$} if $\mu$ is both $\mathbb{P}$-compatible and ergodic with respect to the semigroup $(\Theta^t)_{t \in \mathbb{T}^+}$.
\\ \\
The following is essentially part~(a) of the proof of \cite[Proposition~2]{LeJan87}:

\begin{prop} \label{uniform ergodic}
Let $\mu$ be an ergodic measure of $\varphi$, and let $(\mu_\omega)$ be a disintegration of $\mu$. Then either $(\mu_\omega)$ is atomless or there exists $n \in \mathbb{N}$ such that $(\mu_\omega)$ is $n$-uniform.
\end{prop}

\begin{proof}
Define the function $h\col\Omega \times X \to [0,1]$ by $h(\omega,x)=\mu_\omega(\{x\})$. Note that $h$ is measurable, since it can be expressed as
\[ h(\omega,x) \ = \ \int_X \mathbbm{1}_{\Delta_X\!}(x,y) \, \mu_\omega(dy). \]
\noindent Now for each $t \in \mathbb{T}^+$, let $\Omega_t \subset \Omega$ be a $\mathbb{P}$-full set such that for each $\omega \in \Omega_t$, $\mu_{\theta^t\omega} = \varphi(t,\omega)_\ast\mu_\omega$. Then for all $(\omega,x) \in \Omega_t \times X$, we have
\begin{align*}
h(\Theta^t(\omega,x)) \ &= \ \mu_{\theta^t(\omega)}( \, \{\varphi(t,\omega)x\} \, ) \\
&= \ \mu_\omega \left( \, \varphi(t,\omega)^{-1}(\,\{\varphi(t,\omega)x\}\,) \, \right) \\
&\geq \ \mu_\omega(\{x\}) \\ 
&= \ h(\omega,x).
\end{align*}
\noindent Since $\mu$ is $\mathbb{P}$-compatible, $\mu(\Omega_t \times X)=1$ and so $h \circ \Theta^t \overset{\mu\textrm{-a.s.}}{\geq} h\,$ for each $t \in \mathbb{T}^+$. Hence, since $\mu$ is $(\Theta^t)$-ergodic, there exists $c \in [0,1]$ such that $h^{-1}(\{c\})$ is a $\mu$-full set. So for $\mathbb{P}$-almost every $\omega \in \Omega$, $\mu_\omega$ has the property that $\mu_\omega(\{x\})=c$ for $\mu_\omega$-almost all $x \in X$. It is then clear that either $c=0$ and $(\mu_\omega)$ is atomless, or $c=\frac{1}{n}$ for some $n \in \mathbb{N}$ and $(\mu_\omega)$ is $n$-uniform.
\end{proof}

\noindent Now let $\mathcal{S}$ be the set of probability measures on $X$ that are stationary with respect to the Markov transition probabilities $(\varphi_x^t)$. Let $\mathcal{I}$ be the set of past-measurable invariant measures of $\varphi$. The following is \cite[Theorem~4.2.9]{KS12}:\footnote{In \cite{KS12}, it is assumed that $X$ is Polish, allowing in particular for the result of \cite{doi:10.1080/17442500600745359} to be applied in the construction of the random measure $\mu_\omega$. Nonetheless, in the more general case that $X$ is separable and is Borel in the $d$-completion of $X$, one can regard $X$ (topologically) as a measurable subset of the compact space $[0,1]^\mathbb{N}$; one can then construct the random measure $\tilde{\mu}_\omega$ on $[0,1]^\mathbb{N}$ as the almost sure limit of the sequence of random measures $\mu_\omega^{(n)}(\cdot):=\rho(\varphi(n,\omega)^{-1}( \, \cdot \, \cap X))$, and then (since $\mathbb{E}[\tilde{\mu}_\omega]=\rho(\,\cdot\,\cap X)$) one can take $\mu_\omega$ to be the restriction of $\tilde{\mu}_\omega$ to $\mathcal{B}(X)$.}

\begin{prop} \label{ks12}
$\mathcal{I}$ is mapped bijectively into $\mathcal{S}$ by the mapping $\mathfrak{r}\col\mu \mapsto \pi_{X\ast}\mu$. For any $\rho \in \mathcal{S}$, letting $(\mu_\omega)$ be a disintegration of the past-measurable invariant measure $\mathfrak{r}^{-1}(\rho)$, we have that for any unbounded increasing sequence $(t_n)$ in $\mathbb{T}^+$ there exists a $\mathbb{P}$-full set $\tilde{\Omega} \subset \Omega$ such that for all $\omega \in \tilde{\Omega}$, $\varphi(t_n,\theta^{-t_n}\omega)_\ast\rho$ converges weakly to $\mu_\omega$ as $n \to \infty$.
\end{prop}

\noindent In addition, we have the following:

\begin{prop} \label{unique invariant}
For any $\rho \in \mathcal{S}$, letting $\mu_\rho$ denote the unique past-measurable invariant measure of $\varphi$ satisfying $\pi_{X\ast}\mu_\rho=\rho$, $\mu_\rho$ is also the only $(\Theta^t)$-invariant probability measure on $(\Omega \times X,\mathcal{F} \otimes \mathcal{B}(X))$ whose restriction to $\mathcal{F}_0^\infty \otimes \mathcal{B}(X)$ coincides with $\mathbb{P}|_{\mathcal{F}_0^\infty} \otimes \rho$.
\end{prop}

\begin{proof}
Fix $\rho \in \mathcal{S}$. Let $\mu'$ be any $(\Theta^t)$-invariant probability measure with the property that $\mu'|_{\mathcal{F}_0^\infty \otimes \mathcal{B}(X)}=\mathbb{P}|_{\mathcal{F}_0^\infty} \otimes \rho$. Note that for each $t \in \mathbb{T}^+$, $\Theta^t$ is $(\mathcal{F}_0^\infty \otimes \mathcal{B}(X),\mathcal{F}_{-t}^\infty \otimes \mathcal{B}(X))$-measurable; so then, for any $t \in \mathbb{T}^+$, for all $A \in \mathcal{F}_{-t}^\infty \otimes \mathcal{B}(X)$,
\[ \mu'(A) \ = \ \mu'(\Theta^{-t}(A)) \ = \ \mathbb{P} \otimes \rho(\Theta^{-t}(A)) \ = \ \mu_\rho(\Theta^{-t}(A)) \ = \ \mu_\rho(A). \]
\noindent Since $\mu'$ and $\mu_\rho$ agree on $\mathcal{F}_{-t}^\infty \otimes \mathcal{B}(X)$ for all $t \in \mathbb{T}^+$ and (by assumption) $\mathcal{F}$ is the $\sigma$-algebra generated by $\,\bigcup_{t \in \mathbb{T}^+} \mathcal{F}_{-t}^\infty$, it follows that $\mu'$ and $\mu_\rho$ agree on the whole of $\mathcal{F} \otimes \mathcal{B}(X)$.
\end{proof}

\noindent As an immediate consequence of Propositions~\ref{ks12} and \ref{unique invariant}, we have:

\begin{cor} \label{ergodic correspondence}
For any $\mu \in \mathcal{I}$, $\mu$ is an ergodic measure of $\varphi$ if and only if $\pi_{X\ast}\mu$ is ergodic with respect to $(\varphi_x^t)$.
\end{cor}

\begin{rmk}
The above one-to-one correspondence between past-measurable invariant measures and stationary probability measures is a particular case of a more general one-to-one correspondence between invariant measures and ``forward-time invariant measures'', as described in \cite[Theorem~1.7.2]{Arn98}.
\end{rmk}

\noindent Now we define the \emph{two-point motion} $\varphi \times \varphi \! = \! \left(\varphi \! \times \! \varphi(t,\omega)\right)_{t \in \mathbb{T}^+ \! , \, \omega \in \Omega}\,$ to be the $(\mathbb{T}^+ \! \times \Omega)$-indexed family of functions $\varphi \! \times \! \varphi(t,\omega):X \times X \to X \times X$ given by
\[ \varphi \! \times \! \varphi(t,\omega)(x,y) \ = \ (\varphi(t,\omega)x,\varphi(t,\omega)y). \]
\noindent Note that $\varphi \times \varphi$ is itself a random dynamical system on $X \times X$. We may define the associated family of Markov transition probabilities $(\varphi_{(x,y)}^t)_{x,y \in X \! , \, t \in \mathbb{T}^+}$ by
\[ \varphi_{(x,y)}^t(A) \ = \ \mathbb{P}(\omega : (\varphi(t,\omega)x,\varphi(t,\omega)y) \in A) \hspace{3mm} \forall \, A \in \mathcal{B}(X \times X). \]

\begin{rmk} \label{random prod}
For any invariant measure $\mu$ of $\varphi$, letting $(\mu_\omega)$ be a disintegration of $\mu$, the integrated form of $(\mu_\omega \otimes \mu_\omega)_{\omega \in \Omega}$ is an invariant measure of $\varphi \times \varphi$; so if the invariant measure $\mu$ is past-measurable then (as in \cite[Proposition~2.6(ii)]{Bax91}) the measure $\bar{\mu}^{(2)}$ as defined in Remark~\ref{diagonal measure} is a stationary probability measure of $(\varphi_{(x,y)}^t)$.
\end{rmk}

\begin{sa}
From now on, fix a $(\varphi_x^t)$-ergodic probability measure $\rho$, let $\mu$ be the unique past-measurable ergodic measure of $\varphi$ satisfying $\pi_{X\ast}\mu=\rho$ (with $(\mu_\omega)$ being a disintegration of $\mu$), and let $\bar{\mu}^{(2)}$ be the associated $(\varphi_{(x,y)}^t)$-stationary probability measure on $X \times X$ as described in Remark~\ref{random prod}. For each $\omega \in \Omega$, we define the equivalence relation $\sim_\omega$ on $X$ by
\[ x \sim_\omega y \hspace{4mm} \Longleftrightarrow \hspace{4mm} d(\,\varphi(t,\omega)x \, , \, \varphi(t,\omega)y \,) \, \to \, 0 \, \textrm{ as } t \to \infty. \]
\end{sa}

\noindent Let us now outline the proof of Proposition~\ref{le jan}: It is not hard to show that if $\varphi$ is stable with respect to $\rho$, then $\bar{\mu}^{(2)}(\Delta_X)>0$ and therefore $(\mu_\omega)$ is not atomless;\footnote{cf.~part~(b) of the proof of \cite[Proposition~2]{LeJan87}, or \cite[Lemma~2.17(2)]{FGS14}.} so there exists $n_\rho \in \mathbb{N}$ such that $(\mu_\omega)$ is $n_\rho$-uniform. So $A(\omega):=\mathrm{supp}\,\mu_\omega$ is almost surely an $n_\rho$-element set. Due to (\ref{inv}) and the $(\theta^t)$-invariance of $\mathbb{P}$, we have that for $\mathbb{P}$-almost all $\omega$ the elements of $A(\omega)$ belong to distinct equivalence classes of $\sim_\omega$. Since $\varphi$ is stable with respect to $\rho$, we have that for $\mathbb{P}$-almost all $\omega$, for each $x \in A(\omega)$, the $\,\sim_\omega$-equivalence class of $x$ contains a neighbourhood of $x$. Consequently, as in \cite[Proposition~3]{LeJan87}, one can use the construction of $\mathfrak{r}^{-1}$ in Proposition~\ref{ks12} together with the $(\theta^t)$-invariance of $\mathbb{P}$ to deduce that $\mathbb{P}$-almost all $\omega$, for each $x \in A(\omega)$, the $\,\sim_\omega$-equivalence class of $x$ contains an open set of measure $\frac{1}{n_\rho}$ under $\rho$. (By the second-countability of $X$ and the fact that $\varphi$ is stable with respect to $\rho$, this open set is $\sigma$-contracting under $\omega$.)
\\ \\
We now prove our main result:

\begin{proof}[Proof of Theorem~\ref{main}]
Suppose (i) holds; then since $A_\rho$ is a $\rho$-full set, $A_\rho \cap R \neq \emptyset$, and so there exists $p \in A_\rho$ such that $\mathfrak{C}_p$ contains a $\rho$-full-length rectangle, implying (ii).
\\ \\
Now suppose that (ii) holds. For each $t \in \mathbb{T}^+$, define the map $\Theta_{[2]}^t\col \Omega \times X \times X \to \Omega \times X \times X$ by
\[ \Theta_{[2]}^t(\omega,x,y) \ := \ (\theta^t\omega,\varphi(t,\omega)x,\varphi(t,\omega)y). \]
\noindent Note that the probability measure $\mathbb{P}|_{\mathcal{F}_0^\infty} \otimes \bar{\mu}^{(2)}$ on $(\Omega \times X \times X, \mathcal{F}_0^\infty \otimes \mathcal{B}(X \times X))$ is invariant under the semigroup $(\Theta_{[2]}^t)_{t \in \mathbb{T}^+}$. So by the Poincar\'{e} recurrence theorem,
\[ \mathbb{P} \otimes \! \bar{\mu}^{(2)}( (\omega,x,y) : x \neq y, \, x \sim_\omega y ) \ = \ 0. \]
\noindent Hence, by Fubini's theorem, the set
\[ Y \ := \ \{ \, (x,y) \in (X \times X) \setminus \Delta_X \, : \, \mathbb{P}(\omega: x \sim_\omega y) > 0 \, \} \]
\noindent is an $\bar{\mu}^{(2)}$-null set. Now let $A_1,A_2 \in \mathcal{B}(X)$ be such that $\rho(A_1)>0$, $\rho(A_2)=1$ and $A_1 \times A_2 \subset \mathfrak{C}_{A_\rho}$. We will show that for any $(x,y) \in A_1 \times A_2$, $\mathbb{P}(\omega:x \sim_\omega y)>0$. Fix any $(x,y) \in A_1 \times A_2$, and let $p \in A_\rho$ be such that $(x,y)$ is contractible towards $p$. Let $U,V \subset X$ be open sets with $\bar{U} \subset V$, $\rho(U)>0$ and $\mathbb{P}(E_V)>0$, where $E_V:=\{\omega : \textrm{$V$ contracts under $\omega$}\}$. Let $t_1 \in \mathbb{T}^+$ be such that $\varphi_p^{t_1}(U)>0$. Since $\varphi(t_1,\omega)$ is continuous for all $\omega$, let $r>0$ be such that
\[ k_1:=\mathbb{P}(\omega:\varphi(t_1,\omega)\overline{B_r(p)} \subset \bar{U}) \ > \ 0 \]
and let $t_0 \in \mathbb{T}^+$ be such that
\[ k_0:=\mathbb{P}( \, \omega \, : \, \varphi(t_0,\omega)x,\varphi(t_0,\omega)y \in B_r(p) \, ) \ > \ 0. \]
\noindent Then we have that
\begin{align*}
\mathbb{P}&(\omega : x \sim_\omega y) \\
&\geq \ \mathbb{P}( \, \omega \, : \, \varphi(t_0,\omega)x,\varphi(t_0,\omega)y \in B_r(p) \,\textrm{ and }\, \varphi(t_1,\theta^{t_0}\omega)\overline{B_r(p)} \subset \bar{U} \textrm{ and }\, \theta^{t_0+t_1} \in E_V \, ) \\
&= \ k_0k_1\mathbb{P}(E_V) \\
&> \ 0
\end{align*}
\noindent as required. So in particular, $(A_1 \times A_2) \setminus \Delta_X \, \subset \, Y$. Now since $1=\rho(A_2)=\int_\Omega \mu_\omega(A_2) \, \mathbb{P}(d\omega)$, we have that $\mu_\omega(A_2)=1$ for $\mathbb{P}$-almost all $\omega \in \Omega$, and therefore
\[ \bar{\mu}^{(2)}(A_1 \times A_2) \ = \ \int_\Omega \mu_\omega(A_1)\mu_\omega(A_2) \, \mathbb{P}(d\omega) \ = \ \int_\Omega \mu_\omega(A_1) \, \mathbb{P}(d\omega) \ = \ \rho(A_1). \]
\noindent Let $n \in \mathbb{N}$ be such that $(\mu_\omega)$ is $n$-uniform. By Remark~\ref{diagonal measure}, we have that
\[ \bar{\mu}^{(2)}((A_1 \times A_2) \cap \Delta_X) \ = \ \bar{\mu}^{(2)}(\Delta_{A_1 \cap A_2}) \ = \ \tfrac{1}{n}\rho(A_1), \]
\noindent and therefore
\[ \bar{\mu}^{(2)}\left( (A_1 \times A_2) \setminus \Delta_X \right) \ = \ \tfrac{n-1}{n}\rho(A_1). \]
\noindent But since $(A_1 \times A_2) \setminus \Delta_X \, \subset \, Y$, we also have that
\[ \hspace{-14.6mm} \bar{\mu}^{(2)}\left( (A_1 \times A_2) \setminus \Delta_X \right) \ = \ 0. \]
\noindent Since $\rho(A_1) \neq 0$, it obviously follows that $n=1$, i.e.~(iii) holds.
\\ \\
Now suppose that (iii) holds; we show that (iv) holds. As in Remark~\ref{pairwise sync}, let $A \subset X$ be a $\rho$-full set such that for all $x,y \in A$, $\mathbb{P}(\omega : x \sim_\omega y)=1$; without loss of generality, take $A$ to be a subset of $\mathrm{supp}\,\rho$ such that for each $x \in A$, $\mathbb{P}$-almost every $\omega \in \Omega$ has the property that for all $T \in \mathbb{T}^+$, $\{\varphi(t,\omega)x:t \geq T\}$ is dense in $\mathrm{supp}\,\rho$. Fix any $x,y \in A$, and let $\omega$ be any sample point with the properties that $x \sim_\omega y$ and for all $T \in \mathbb{T}^+$, $\{\varphi(t,\omega)x:t \geq T\}$ is dense in $\mathrm{supp}\,\rho$. Fix any open $U \subset X$ with $\rho(U)>0$, and let $p \in U$ and $\varepsilon>0$ be such that $B_\varepsilon(p) \subset U$. Let $T \in \mathbb{T}^+$ be such that for all $t \geq T$, $d(\varphi(t,\omega)x,\varphi(t,\omega)y)<\frac{\varepsilon}{2}$; and let $t' \geq T$ be such that $\varphi(t',\omega)x \in B_{\frac{\varepsilon}{2}}(p)$. Then both $\varphi(t',\omega)x$ and $\varphi(t',\omega)y$ are in $B_\varepsilon(p)$ and hence in $U$.
\\ \\
Finally, it is clear that (iv)$\Rightarrow$(i) (with $R=\mathrm{supp}\,\rho$).
\end{proof}

\sloppy
\printbibliography[maxnames=99]

\end{document}